\documentclass[11pt]{amsart}

\usepackage{amsmath, amssymb, amsfonts,enumerate}
\usepackage[all]{xy}
\usepackage{amscd}
\usepackage{comment}
\usepackage{mathtools}
\usepackage{tikz} 
\usepackage{tikz-cd} 
\tikzcdset{diagrams={nodes={inner sep=7.5pt}}}

\usepackage{url}
\usepackage{hyperref}

\newtheorem{theorem}{Theorem}[section]
\newtheorem{lemma}[theorem]{Lemma}
\newtheorem{corollary}[theorem]{Corollary}
\newtheorem{proposition}[theorem]{Proposition}
\newtheorem{theoremletter}{Theorem}

 \theoremstyle{definition}
 \newtheorem{definition}[theorem]{Definition}

\newtheorem{examples}[theorem]{Examples}
\newtheorem{question}{Question}

\numberwithin{equation}{section}
 
\newcommand {\N}{\mathbb{N}} 
\newcommand {\Z}{\mathbb{Z}} 
 
\newcommand {\Q}{\mathbb{Q}} 
 
\newcommand{\F}{\mathbb{F}}

\newcommand{\CC}{\mathcal{C}}

 \newcommand{\Set}{\mathrm{Set}}

\DeclareMathOperator{\Id}{Id}

\begin{document}

\title[LEF-groups and endomorphisms of symbolic varieties]{LEF-groups and  endomorphisms of symbolic varieties}
\author[Xuan Kien Phung]{Xuan Kien Phung}
\email{phungxuankien1@gmail.com}
\begin{abstract}
Let $G$ be a group and let $X$ be an algebraic variety over an algebraically closed field $k$ of characteristic zero. Denote $A=X(k)$ the set of rational points of $X$. We investigate invertible algebraic cellular automata $\tau \colon A^G \to A^G$ whose local defining map is induced by some morphism of algebraic varieties $X^M \to X$ where $M$ is a finite memory. When $G$ is locally embeddable into finite groups (LEF), then we show that the inverses of reversible  algebraic cellular automata are automatically algebraic cellular automata. Generalizations are also obtained for finite product Hopfian pointed object alphabets in concrete categories. 
\end{abstract} 
\date{\today}
\maketitle

\setcounter{tocdepth}{1}

\section{Introduction} 

Let us recall briefly notions of symbolic dynamics. Fix a set $A$ called the \emph{alphabet}, and a group  $G$, the \emph{universe}.
A \emph{configuration} $c \in A^G$ is simply a map $c \colon G \to A$. 
The Bernoulli shift action of  $G$ on $A^G$ is  $(g,c) \mapsto g c$, 
where $(gc)(h) \coloneqq  c(g^{-1}h)$ for  $g,h \in G$ and $c \in A^G$. 
Introduced by von Neumann \cite{neumann}, a \emph{cellular automaton} over $G$ and $A$ is a map
$\tau \colon A^G \to A^G$ with a \emph{local defining map} $\mu \colon A^M \to A$ for some finite \emph{memory set} $M \subset G$
and such that 
\begin{equation*} 
\label{e:local-property}
(\tau(c))(g) = \mu((g^{-1} c )\vert_M)  \quad  \text{for all } c \in A^G \text{ and } g \in G.
\end{equation*} 
\par 
Now let $X$  be an  algebraic variety over a field $k$, i.e., a reduced separated $k$-scheme of finite type (cf.~\cite{grothendieck-ega-1-1}). Denote by  $X(k)$ the set of rational points of $X$. The set $CA_{alg}(G,X,k)$ of \emph{algebraic cellular automata} consists of cellular automata $\tau \colon X(k)^G \to X(k)^G$ which  admit a memory  $M\subset G$ and 
a local defining map $\mu \colon X(k)^M \to X(k)$ induced by some $k$-morphism of algebraic varieties 
$f \colon X^M \to X$, i.e., $\mu=f\vert_{X(k)^M}$, 
where $X^M$ is the fibered product of copies of $X$ indexed by $M$. 
\par 
A cellular automaton $\tau \colon A^G \to A^G$ is \emph{reversible} if it is bijective and the inverse map $\tau^{-1} \colon A^G \to A^G$ is also a cellular automaton. It is well-known that if the alphabet $A$ is finite then every bijective cellular automaton $\tau \colon A^G \to A^G$ is reversible (cf., e.g. \cite[Theorem~1.10.2]{ca-and-groups-springer}). 
\par 
Generalizing the result for finite alphabets, \cite[Theorem~1.3]{cscp-alg-ca} shows that over an uncountable algebraically closed base field, every bijective algebraic cellular automaton is reversible. It is then natural to ask: 
\begin{question}
In characteristic zero, are the inverses of reversible algebraic cellular automata also  algebraic cellular automata? 
\end{question}  
\par 
The following example shows us what can go wrong in positive characteristic. Let $\mathbb{A}^1$ be the affine line and let $k$ be the algebraic closure of a finite field $\F_p$ for some prime number $p$. Then for every group $G$, the algebraic cellular automaton $\tau \colon k^G \to k^G$ given by $\tau(x)(g)=x(g)^p$ for all $x \in k^G$, $g \in G$, is reversible but it inverse is clearly not an  algebraic cellular automaton. 
\par 
We have the following result proved in \cite[Theorem~1.4]{cscp-alg-ca}: 

\begin{theorem}
\label{t:invertible-intro}
Let $G$ be a locally residually finite group. Let $X$ be an algebraic variety over an algebraically closed field $k$ of characteristic zero. 
Suppose that $\tau \in CA_{alg}(G,X,k)$ is reversible. Then $\tau^{-1} \in CA_{alg}(G,X,k)$.  
\end{theorem} 

\par 
Here, recall that a group $G$ is  residually finite if the intersection of the finite-index
subgroups of $G$ reduces to the identity element. For example, locally finite groups,   finitely generated abelian groups, and  finitely generated linear groups are residually finite by a theorem of Mal'cev.  
\par 
The first goal of this present paper is to generalize Theorem~\ref{t:invertible-intro} to the class of groups that are  locally embeddable into finite groups which are also known as initially subfinite groups in the terminology of Gromov.   

\begin{definition} 
A subset $A$ of a group $G$ is said to be \emph{embeddable into a group $H$} if there exists a one-to-one map $\varphi \colon A \to H$ such that for every $a,b \in A$ with $ab \in A$, we have $\varphi(ab)=\varphi(a) \varphi(b)$. 
A group $G$ is  \emph{locally embeddable into finite groups} (LEF), or an \emph{LEF-group}, if every finite subset of $G$ is embeddable into a finite group. 
\end{definition}
\par 
Equivalently, a group is LEF if and only if it is isomorphic to a subgroup of an ultraproduct of finite groups. LEF-groups were  introduced and studied by 
Vershik and Gordon in \cite{lef-group} by  extending the general concepts for algebraic structures introduced by 
Mal'cev in \cite{malcev-lef}, \cite{malcev-lef-2}. 
\par 
We know that 
all locally residually finite groups and all locally residually amenable groups are LEF-groups. While  finitely presented LEF-groups are residually finite, there exist finitely generated amenable groups which are LEF but not residually finite \cite{lef-group}. On the other hand, all LEF-groups are sofic (cf.~\cite{gromov-esav}, \cite{weiss-sgds}) and no finitely presented infinite simple group is a  LEF-group.  
\par 
The main result of the paper is the following  generalization of  Theorem~\ref{t:invertible-intro} to cover the larger class of LEF-group universes: 

\begin{theoremletter}
\label{t:main-intro-invertible} 
Let $G$ be a LEF-group and let $X$ be an algebraic variety over an algebraically closed field $k$ of characteristic zero. 
Suppose that $\tau  \in~CA_{alg}(G,X,k)$ is reversible. Then one has $\tau^{-1} \in CA_{alg}(G,X,k)$. 
\end{theoremletter}
\par 
To further generalize Theorem~\ref{t:main-intro-invertible}, one may expect that the class of groups that are locally embeddable into locally residually finite groups is larger  than the class of LEF-groups since every finite group is obviously residually finite. 
\par However, the two classes of groups are in fact  equal. Indeed, 
suppose that a finite subset $A$ of a group $G$ is embeddable into a residually finite group $H$ via an injective map $\varphi \colon A \to H$ with  $\varphi(ab)=\varphi(a) \varphi(b)$ for all $a,b \in A$ such that $ab \in A$. Then, since $H$ is residually finite, we can clearly find a finite group $F$ and a homomorphism of groups $\psi \colon H \to F$ which separates the finite set $\{ \varphi(a)\colon a \in A\}$, i.e., $\psi(\varphi(a)) \neq \psi(\varphi(b))$ for all distinct $a,b \in A$. It is immediate that the composition $\psi \circ \varphi \colon A \to F$ is one-to-one and embeds $A$ into the finite group $F$. 
\par 
As another motivation for Theorem~\ref{t:main-intro-invertible}, we recall a general invertibility  result of bijective endomorphisms of symbolic   group varieties. 
\par 
For the notations, let $G$ be a group and let $X$ be an algebraic group over a field $k$. Then the class $CA_{algr}(G,X,k)$ consists of cellular automata $X(k)^G \to X(k)^G$ which  admit for some finite set $M \subset G$ 
a local defining map $\mu \colon X(k)^M \to X(k)$ induced by some $k$-homomorphism of algebraic groups  
$f \colon X^M \to X$, i.e., $\mu=f\vert_{X(k)^M}$. It is clear from the definition that   $CA_{algr}(G,X,k) \subset CA_{alg}(G,X,k)$. 
It is shown in \cite[Theorem~6.4]{phung-2020} that: 
\par

\begin{theorem}
\label{t:intro-general-ca-alg-invertible} 
Let $G$ be a group and let $X$ be an algebraic group over an algebraically closed $k$ of characteristic zero. Suppose that $\tau \in CA_{algr}(G,X,k)$ is bijective. Then one has $\tau^{-1} \in CA_{algr}(G,X,k)$. 
\end{theorem}

As a related result of the  surjunctivity and the reversibility, we also obtain a short proof of the following result (cf.~Section~\ref{s:reversibility}) for endomorphisms of symbolic varieties in arbitrary characteristic by adopting the proof of \cite[Theorem~A]{phung-weakly} using scheme theory.  

\begin{theoremletter}
\label{t:general-surjunctive} 
Let $G$ be a surjunctive group and let $X$ be an algebraic variety over an uncountable algebraically closed field $k$. Then every injective endomorphism  $\tau\in CA_{alg}(G,X,k)$ is reversible. 
\end{theoremletter}
\par 
Here, a group $G$ is \emph{surjunctive} if for every finite set $A$, every injective cellular automaton $A^G\to A^G$ is surjective. Surjunctive groups were first introduced by Gottschalk \cite{gottschalk} and it is known that all residually finite groups, all amenable groups, and more generally all sofic groups are surjunctive (see~\cite{{gromov-esav}}, \cite{weiss-sgds}, notably \cite[\S 7.G]{gromov-esav} for much more general results). 
\par 
In the above theorem, if $G$ is only supposed to be \emph{weakly surjunctive}, i.e., for every finite group $A$, every injective cellular automaton $A^G \to A^G$ which is also a homomorphism of abstract groups is automatically bijective, and if $X$ is an algebraic group and $\tau \in CA_{algr}(G,X,k)$, then \cite[Theorem~A]{phung-weakly} shows that the same conclusion of Theorem~\ref{t:general-surjunctive} holds. 
\par 
The paper is organized as follows. We recall basic properties of algebraic varieties in Section~\ref{s:preliminary}. We also recall the notion of induced local maps of cellular automata that allows us to describe a simple but useful local criterion for one-sided invertible cellular automata. In Section~\ref{s:invertibility}, we present the proof of the main result (Theorem~\ref{t:main-intro-invertible}). Then Section~\ref{s:reversibility} contains a short proof of Theorem~\ref{t:general-surjunctive} following the method of the proof of \cite[Theorem~A]{phung-weakly}. 
In Section~\ref{s:ca-concrete-category}, we introduce and investigate the notions of \emph{finite product Hopfian} and \emph{finite product co-Hopfian} objects as well as cellular automata over such alphabets.  
We then formulate in  Section~\ref{s:generalization} a generalization of Theorem~\ref{t:main-intro-invertible} for reversible and one-sided invertible cellular automata over finite product Hopfian and finite product co-Hopfian pointed alphabets in a concrete category. Finally, we give some direct applications and examples in Section~\ref{s:applications}. 

\section{Preliminaries}
\label{s:preliminary}

\subsection{Models of morphisms of finite type}
\label{s:model-finite-data} 

We shall need the following   auxiliary lemma in algebraic geometry for the proof of Theorem~\ref{t:general-surjunctive}:  
\begin{lemma}
\label{l:model-finite-data}
Let $X, Y$ be algebraic varieties over a field $k$. Let $f_i \colon X^{m_i} \to Y^{n_i}$, $m_i , n_i \in \N$, $i \in I$, be finitely many morphisms of $k$-algebraic varieties. Then there exist a finitely generated $\Z$-algebra $R \subset k$ and $R$-schemes of finite type $X_R$, $Y_R$ and $R$-morphisms $f_{i,R} \colon (X_R)^{m_i} \to (Y_R)^{n_i}$ of $R$-schemes with $X=X_R \otimes_R k$, $Y= Y_R \otimes_R k$, and $f_i=f_{i,R} \otimes_R k$ (base change to $k$). Moreover, if $X=Y$, one can take $X_R=Y_R$ and if $f_i$ is a closed immersion, one can also choose $f_{i, R}$ to be a closed immersion. 
\end{lemma}

\begin{proof}
See, e.g., \cite[Section~8.8]{ega-4-3}, in particular  \cite[Scholie~8.8.3]{ega-4-3}, and  \cite[Proposition~8.9.1]{ega-4-3}. 
\end{proof}

\subsection{Induced local maps} 
\label{s:induced-map}
For the notations, let $G$ be a group and let $A$ be a set. Let $\tau \colon A^G \to A^G$ be a cellular automaton. Fix a  memory set $M$ and the corresponding  local defining map $\mu \colon A^M \to A$. For every finite subset $E \subset G$, we denote by $\tau_E^+ \colon A^{E M } \to A^E$ the induced local map of $\tau$ by setting 
$\tau_E^+(x)(g) = \mu ((g^{-1}(x))\vert_M)$ for all $x \in A^{EM}$ and $g \in E$. 
\par 
We have the following simple lemma: 

\begin{lemma}
\label{l:direct-local}
Let $G$ be a group and let $A$ be a set. Let $\tau, \sigma \colon A^G \to A^G$ be respectively cellular automata with local defining maps $\mu, \eta \colon A^M \to A$ for some  common memory set $M \subset G$ such that $1_G \in M$. Denote by $\pi \colon A^{M^2} \to A^{\{1_G\}}$ the canonical projection. Then the following hold: 
\begin{enumerate}[\rm (i)]
    \item $\sigma \circ \tau= \Id$ if and only if $\eta \circ \tau_M^+= \pi$;
    \item 
    $\tau \circ \sigma= \Id$ if and only if $\mu \circ \sigma_M^+= \pi$.
\end{enumerate}
\end{lemma}

\begin{proof}
Suppose first that $\sigma \circ \tau = \Id$. Let $x \in A^{M^2}$ and let $y \in A^G$  be an arbitrary configuration such that $y \vert_{M^2}=x$. We find from the very definition of local defining maps that: 
\begin{align*}
\eta ( \tau_M^+(x)) & = \eta(\tau(y)\vert_M) \\
& = (\sigma (\tau(y)))(1_G)
\\ & = y(1_G) = \pi(x) 
\end{align*}
\par 
It follows that $\eta \circ \tau_M^+= \pi$. 
Conversely, suppose that  $\eta \circ \tau_M^+= \pi$. Then for every $y \in A^G$ and $g \in G$, we find from the $G$-equivariance of $\tau$ and $\sigma$ that: 
\begin{align*}
(\sigma (\tau(y)))(g) & = 
\eta((g^{-1}\tau(y))\vert_M)
\\& = \eta (\tau(g^{-1}y)\vert_M)
\\& = \eta (\tau_M^+((g^{-1}y)\vert_{M^2}))
\\& = \pi(((g^{-1}y))\vert_{M^2})\\& = y(g). 
\end{align*}
\par 
Therefore, we deduce that $\sigma \circ \tau = \Id$. This terminates the proof of the point (i). The proof of (ii) is completely similar. 
\end{proof}

\subsection{Restriction cellular automata}
In order to reduce to the case of finitely generated group universes, we shall need the following useful lemma. 

\begin{lemma}
\label{l:induction-restriction} 
Let $G$ be a group and let $A$ be a set. Let $\tau \colon A^G \to A^G$ be a cellular automaton. Let $\mu \colon A^M \to A$ be a local defining map of $\tau$ for some finite memory $M \subset G$.  Let $H$ be the subgroup generated by $M$ and let $\tau_H \colon A^H \to A^H$ be the cellular automaton which also admits $\mu \colon A^M \to A$ as a local defining map. Then $\tau$ is injective, resp. surjective, if and only if $\tau_H$ is injective, resp. surjective. 
\end{lemma}

\begin{proof}
See the main result of \cite{csc-induction}.
\end{proof}

\section{LEF-groups and invertibility of reversible endomorphisms of symbolic varieties} 
\label{s:invertibility}
The goal of the section is to give a proof of Theorem~\ref{t:main-intro-invertible}. To recall the notations, let $G$ be a LEF-group and let $X$ be an algebraic variety over an algebraically closed field $k$ of characteristic zero. 
\par 
Suppose that we are given a reversible   $\tau  \in~CA_{alg}(G,X,k)$. We need to show that $\tau^{-1}\in CA_{alg}(G,X,k)$ as well.  

\begin{proof}[Proof of Theorem~\ref{t:main-intro-invertible}] 
Since $\tau$ is reversible by hypothesis, its inverse $\tau^{-1}$ is also a cellular automaton. Therefore, we can choose a finite symmetric set $M \subset G$ containing $1_G$ such that $M$ is a memory set of both $\tau$ and $\tau^{-1}$. 
Let $\mu \colon X^M \to X$ be the morphism of algebraic varieties which serves as the local defining map of $\tau$. Similarly, we have a local defining map $\eta \colon X(k)^M \to X(k)$ of $\tau^{-1}$. Our goal is to prove that $\eta$ is induced by a morphism of algebraic varieties $X^M \to X$.  
\par 
By the universal property of fibered products, we can define a morphism $\tau_M^+ \colon X^{M^2} \to X^M$ of algebraic varieties given by $\tau_M^+(z)(g)= \mu((g^{-1}z)\vert_M)$ for all $z \in X^M$ and $g \in M$. Let $\pi\colon  A^{M^2} \to A^{\{1_G\}}$ be the canonical  projection given by $v \mapsto v(1_G)$. Since $\tau^{-1} \circ \tau=\Id$, we infer from Lemma~\ref{l:direct-local} the relation $\eta\circ \tau_M^+=\pi$. 
\par 
Now, since $G$ is an LEF-group, we can find a finite group $F$ and an embedding  $\varphi \colon M^2 \to F$ such that $\varphi(ab)= \varphi(a) \varphi(b) $ for all $a,b \in M$ (note that $ab\in M^2$).  Moreover, up to  restricting to the subgroup of $F$ generated by the images $\varphi(a)$ for $a \in M$, we can clearly suppose without loss of generality that $F$ is generated by the finite set $E=\{\varphi(a) \colon a \in M\}=\varphi(M)$. 
\par 
Observe that since $M$ is symmetric and contains $1_G$ and since $\varphi$ is an embedding, $E= \varphi(M)$ is also a finite symmetric subset of $F$ and that $1_F=\varphi(1_G) \in E$. Note also that $E^2=\varphi(M^2)$. 
\par 
We define a trivial isomorphism of algebraic varieties $\delta_E \colon X^M \to X^E$ given by $x  \mapsto y$ where  $y(h)=x(\varphi^{-1}(h))$ for all $x \in X^M$ and $h \in E$. 
\par 
Consider the morphism of algebraic varieties $f \colon X^E \to X$ induced by $\mu$ via the reindexing bijection $M \to E$, $p \mapsto \varphi(p)$, that is, $f= \mu \circ \delta_{E}^{-1}$. More concretely, we set $f(x)=\mu(y)$ for all $x \in X^E$, $y\in X^M$ with $y(g)= x(\varphi(g))$ for all $g \in M$. 
\par 
By the universal property of fibered products, the morphism $f$ in turn induces an $F$-equivariant morphism of algebraic varieties 
\[
\alpha \colon X^F \to X^F
\]
defined by $\alpha(x)(h)= f((h^{-1}x)\vert_E)$ for all $x \in X^F$ and $h \in F$. 
\par 
Similarly, the map $\eta$ induces a set map $\psi \colon X(k)^E \to X(k)$ via the bijection $M \to E$, $p \mapsto \varphi(p)$  and another $F$-equivariant map 
\[
\beta \colon X(k)^F \to X(k)^F
\]
defined by $\beta(x)(h)= \psi((h^{-1}x)\vert_E)$ for all $x \in X^F$ and $h \in F$. 
\par 
We claim that $\beta \circ \alpha = \Id_{X(k)^F}$. Indeed, 
consider the morphism of algebraic  varieties $\alpha_E^+ \colon X^{E^2} \to X^E$  given by $\alpha_E^+(x)(h)= f((h^{-1}x)\vert_E)$ for all $x \in X^E$ and $h \in E$. Let $\rho \colon X^{E^2} \to X^{\{1_F\}}$ be the canonical projection $u \mapsto u(1_G)$. 
\par 
Observe that  
since $\eta\circ \tau_M^+=\pi$ and since $\varphi$ embeds $M^2$ into $F$, we deduce from our construction that $\psi \circ \alpha_E^+=\rho$. 
Therefore, by our definition of the maps $\alpha$ and $\beta$, it follows immediately that for every $x \in X(k)^F$ and $h \in F$, we have: 
\begin{align*}
    (\beta(\alpha(x)))(h) & = \psi((h^{-1}\alpha(x))\vert_E)\\
    & = \psi((\alpha(h^{-1}x))\vert_E)\\
    &=\psi(\alpha_E^+((h^{-1}x)\vert_{E^2}))\\ 
    & = (h^{-1}x)(1_F)\\
    &= x(h).
\end{align*}
\par 
Consequently, we find that $\beta \circ \alpha= \Id$ and the claim is proved. In particular, we deduce that the restriction of $\alpha$ to the set $X(k)$ is injective. 
\par 
Since $k$ is algebraically closed, it follows that $\alpha$ is  injective as a morphism of algebraic varieties (cf., e.g.~\cite[Lemma~A.20, Lemma~A.22]{cscp-alg-ca}). 
As the base field $k$ has  characteristic zero, the main result of Nowak in  \cite{nowak} shows that the morphism $\alpha$ is in fact an  automorphism of algebraic varieties. 
\par In particular, it follows that the map $\beta$ is induced by a morphism of algebraic varieties that we denote by $\gamma \colon X^F \to X^F$. 
\par 
Fix $c \in X(k)^{F\setminus M}$ and let $\iota \colon X^E \to X^F$ be the closed  immersion given by $x \mapsto (x,c)$. Let $\omega \colon X^F \to X$ denote the canonical projection $x \mapsto x(1_F)$. 
\par 
Therefore, we obtain a morphism of algebraic varieties $\nu \colon X^M \to X$ given by the composition: 
\begin{equation*}
   \nu \coloneqq  \omega \circ \gamma \circ \iota \circ \delta_E, 
\end{equation*}
that is, we have 
$\nu (x)=(\gamma(\iota (\delta_E(x))))(1_F)$ for all $x \in X^M$. 
\par 
By construction, it is clear that $\eta= \nu\vert_{X(k)^M}$ and thus we can conclude that $\tau^{-1} \in CA_{alg}(G,X,k)$. 
The proof of the theorem is complete. 
\end{proof}

\section{Reversibility of injective endomorphisms of symbolic varieties} 
\label{s:reversibility}
To recall the notations for the proof of Theorem~\ref{t:general-surjunctive}, let $G$ be a surjunctive group and let $X$ be a reduced scheme of finite type over an uncountable algebraically closed field $k$. We fix a finite memory set $M \subset G$ of an injective $\tau \in CA_{alg}(G,X,k)$ such that $1_G \in M$. We need to show that $\tau$ is bijective and reversible. 

\begin{proof}[Proof of Theorem~\ref{t:general-surjunctive}]
Suppose first that $G$ is countable so that we can find an exhaustion $(E_n)_{n \in \N}$ of $G$ such that $1_G \in E_0$ and $G= \cup_{n \in \N} E_n$. 
\par 
For every $n \in \N$, 
we have a 
$k$-morphism $\tau_{E_n}^+ \colon X^{E_n M} \to X^{E_n}$ of algebraic varieties defined in Section \ref{s:induced-map} and a $k$-morphism of algebraic varieties 
\begin{equation}
\label{e:diagonal-reversible}
\Phi_n \coloneqq \tau_{E_n}^+ \times \tau_{E_n}^+ \colon  X^{E_nM}\times_k  X^{E_n M}\to  X^{E_n}\times_k X^{E_n }.
\end{equation} 
\par 
Let us define $\pi_{n} \colon X^{E_n M}\times_k  X^{E_n M}\to X^{\{1_G\}} \times_k  X^{\{1_G\}}$ to be the canonical projection. 
For every finite subset $E \subset G$, the diagonal of $X^E \times_k X^E$ is denoted by $\Delta_E$.  
Let us consider the  constructible subset of $X^{E_n M}\times_k  X^{E_n M}$: 
\begin{equation} 
V_n \coloneqq \Phi_n^{-1}(\Delta_{E_n})\setminus \pi_{n}^{-1}(\Delta_{\{1_G\}}).
\end{equation} 
\par 
Observe that closed points of $\Phi_n^{-1}(\Delta_{E_n})$ are 
 pairs $(u,v)\in  A^{E_n M}\times A^{E_n M}$ such that $\tau_{E_n}^+(u)=\tau_{E_n}^+(v)$. 
Similarly, the set of closed points of 
$\pi_{n}^{-1}(\Delta_{1_G})$ consists of pairs  
$(u,v)\in  A^{E_n M}\times A^{E_n M}$ such that $u(1_G)= v(1_G)$. 
\par 
We claim that there exists $n \in \N$ such that $V_n = \varnothing$. Indeed, suppose on the contrary that $V_n \neq \varnothing$ for every $n \in \N$. 
For every $m \geq n \geq 0$, 
we have a canonical projection 
\[
p_{m,n} \colon X^{E_m M} \times X^{E_m M} \to X^{E_n M}\times  X^{E_n M}
\] which is clearly a $k$-morphism of algebraic varieties such that 
$p_{m,n}(V_m)\subset V_n$. 
\par 
Consequently, we obtain an inverse system 
$(V_n,p_{m,n})_{m \geq n \geq 0}$ of nonempty constructible subsets of $k$-algebraic varieties with algebraic transition morphisms. 
Since the base field $k$ is uncountable, we infer from  \cite[Lemma~B.2]{cscp-alg-ca} or \cite[Lemma~3.2]{cscp-invariant-ca-alg} that $\varprojlim V_n \neq \varnothing$. Note that since  
\begin{equation}
    \varprojlim_n V_n \subset \varprojlim_n A^{E_n M} \times A^{E_n M} = A^G \times A^G, 
\end{equation}
we can find $x,y \in A^G \times A^G$ such that $\tau(x)= \tau(y)$ and $x(1_G) \neq y(1_G)$ by the description of the sets $V_n$. Therefore, $x \neq y$ and $\tau$ cannot be injective. 
\par 
This contradiction proves the claim that $V_n= \varnothing$ for some $n \geq 0$ that we fix in what follows. 
Since $V_n = \varnothing$, we deduce that $W\coloneqq \Phi_n^{-1}(\Delta_{E_n})$ is a closed subvariety of $U \coloneqq  \pi_{n}^{-1}(\Delta_{1_G})$. 
\par 
By Lemma~\ref{l:model-finite-data}, there exists a finitely generated $\Z$-algebra $R \subset k$ and an  $R$-scheme of finite type $X_R$ and a morphism of $R$-schemes: 
\[
 \mu_R \colon (X_R)^M \to X_R 
\] 
such that $X= X_R \otimes_R k$ and  $\mu= \mu_R \otimes_R k$ and satisfy   the following properties. 
\par 
Denote by $T_R \colon (X_R)^{E_n M } \to (X_R)^{E_n}$ the morphism of $R$-schemes defined by the universal property of fibered products by the collection of $R$-morphisms  $(T_{g})_{g \in E_n}$ where the $R$-morphism $T_{E,g} \colon (X_R)^{E_n M} \to (X_R)^{\{g\}}$, for $g \in E_n$, is the composition of the projection $(X_R)^{E_nM} \to (X_R)^{g M}$ followed by the $R$-morphism $(X_R)^{gM} \to (X_R)^{\{g\}}$ induced by $\mu$ via the isomorphism $(X_R)^{gM} \simeq (X_R)^M$ and $(X_R)^{\{G\}}\simeq (X_R)^{\{1_G\}}$ given by the reindexing bijection $M \to gM$, $h \mapsto gh$. In particular, we find that 
\[
\tau_{E_n}^+ = T_R \otimes_R k. 
\]
\par 
Moreover, if we denote by $\pi_R$  the canonical projection morphism of $R$-schemes $ (X_R)^{E_n M}\times_R  (X_R)^{E_n M}\to (X_R)^{\{1_G\}} \times_R  (X_R)^{\{1_G\}}$ 
and define the $R$-morphism 
\begin{equation}
\label{e:diagonal-reversible}
\Phi_R \coloneqq T_R \times T_R \colon  (X_R)^{E_n M}\times_R  (X_R)^{E_n M}\to  (X_R)^{E_n}\times_R (X_R)^{E_n}, 
\end{equation}
then we can choose $R$ such that $W_R= \Phi_R^{-1}(\Delta_{E_n, R})$ is a closed $R$-subscheme of $U_R= \pi_R^{-1}(\Delta_{\{1_G\}, R})$. 
Note that $W=W_R \otimes_R k$ and  $U=U_R \otimes_R k$. 
\par 
Now, the exact same proof of  \cite[Theorem~4.2]{phung-weakly} shows that $\tau$ is surjective. The only obvious modification needed in the proof of \cite[Theorem~4.2]{phung-weakly} is that the finite set $H_{p,s,d}$ is no longer a group and thus $\tau_{p,s,d}$ is no longer a group cellular automaton. But as  $\tau_{p,s,d}$ is injective, it is also surjective since $G$ is surjunctive by hypothesis. Apart from this remark, the rest of the proof of the surjectivity of $\tau$ is identical. 
\par 
To conclude, it suffices to apply \cite[Lemma~4.3]{phung-weakly} to see that there exists an arbitrary large  finite subset $N \subset G$ and a map $\eta \colon A^N \to A$  such that for all $x \in A^{N M}$, we have $\eta ( \tau_N^+(x))=x(1_G)$. In particular, we deduce from Lemma~\ref{l:direct-local} that $\tau$ is reversible whose inverse $\tau^{-1}\colon A^G \to A^G$ admits $\eta$ as a local defining map. 
\par 
For the general case where $G$ is not necessarily countable, let $H$ be the subgroup of $G$ generated by the finite set $M$. Let $\tau_H \colon A^H \to A^H$ be the restriction cellular automaton which also admits $\mu \colon X^M \to X$ as a local defining map. Then $\tau_H$ is also injective by Lemma~\ref{l:induction-restriction}. Therefore, we have seen in the above that $\tau_H$ must be surjective as $H$ is finitely generated. But Lemma~\ref{l:induction-restriction}  implies that $\tau$ is surjective as well. The proof is thus complete.  
\end{proof}

\section{Cellular automata over pointed object alphabets in concrete categories}  \label{s:ca-concrete-category}

\subsection{Hopfian and co-Hopfian objects} 

Let us fix a category $C$. In general, we say that an object $A$ is a \emph{Hopfian object} if every epimorphism of $A \to A$ is automatically an automorphism. 
Similarly, an object $A$ in $\CC$ 
is \emph{co-Hopfian} if every monomorphism from $A\to A$ is  an automorphism. 
\begin{examples} 
\label{ex:hopfian}
We have the following general examples of Hopfian and co-Hopfian objects (see also \cite{varadarajan}):  
\begin{enumerate}
    \item 
Every Noetherian module is Hopfian, and every Artinian module is co-Hopfian as a module.
\item 
Every one sided Noetherian or Artinian ring is a Hopfian ring. 
\item
Examples of Hopfian groups include  finite groups, and more generally polycyclic-by-finite groups, finitely-generated free groups, 
finitely generated residually finite groups, and torsion-free word-hyperbolic groups. 
\item 
Observe that the additive group $(\Q,+)$ of rational numbers is Hopfian while the additive group of real numbers is not.
\item 
Co-Hopfian groups admit more geometric examples: the mapping class group of a closed hyperbolic surface and the fundamental group of a closed aspherical manifold with nonzero Euler characteristic are co-Hopfian. 
\end{enumerate}
\end{examples} 
\par 

\begin{definition}
We say that an object $A\in \CC$ is \emph{finite product Hopfian}, resp. \emph{finite product co-Hopfian}, if for every $n \in \N$, the object $A^n \in \CC$ is  Hopfian, resp. co-Hopfian. 
\end{definition}

As an example, we have the following class of finite product Hopfian and finite product co-Hopfian objects. 

\begin{proposition}
\label{p:finite-product-commutative-ring-hopfian}
Let $R$ be a commutative ring. Then every finitely generated $R$-module $M$ is finite product co-Hopfian as an $R$-module. 
\end{proposition}

\begin{proof}
Since the ring $R$ is commutative by hypothesis, we infer from the results of \cite{strooker} and  \cite{vasconcelos} that every finitely generated $R$-module is co-Hopfian. On the other hand, observe that $M^n$ is clearly a finitely generated $R$-module for every $n \in \N$ as $M$ is finitely generated. Therefore, $M^n$ is co-Hopfian as an $R$-module. We can thus conclude that $M$ is a finite product co-Hopfian  $R$-module. The proof is complete. 
\end{proof}

\subsection{Pointed objects and concrete categories}  
We say that a category $\CC$ is a \emph{concrete category} if $\CC$ admits a faithful functor $F \colon \CC \to \Set$ from $\CC$ to the category of sets.  \par 
Intuitively, we can think of  the functor $F$ as the forgetful functor which associates with every object $A$ of $\CC$ the underlying set $F(A)$, and similarly with  every morphism $f \colon A \to B$ in $\CC$ the underlying mapping on sets  $F(f) \colon F(A) \to F(B)$. Note that since $F$ is faithful, every morphism $f$ in $\CC$ is completely determined by the underlying set map $F(f)$. 
\par 
We will reserve the terminology \emph{morphism} for morphisms in a given category while we simply say \emph{maps} or \emph{set maps} the functions of sets. When the context is clear, we use the same notations $A, B$ and $f \colon A \to B$ to refer to the underlying sets $F(A), F(B)$ and the set map $F(f)$. 
\par 
Now suppose that $\CC$ is a category with a \emph{terminal object} that we denote by $\varepsilon_\CC$. An object $A \in \CC$ is called a \emph{pointed object} equipped with a morphism $a \colon \varepsilon_\CC \to A$. We define a morphism of pointed objects $(A, a) \to (B,b)$ to be a morphism $f \colon A \to B$ such that $ f \circ a=b$.

\subsection{The class $CA_{\CC}(G,(A,a))$}  
Let us fix a group $G$ and a concrete category $\CC$ with finite fibered products and a distinguished  terminal object $\varepsilon$. 
\par 
Hence, by the universal property of fibered products, we find that for all finite sets $E \subset F$ and every pointed object $(A,a) \in \CC$, the canonical projections $\pi \colon A^F \to A^E$, $x \mapsto x\vert_E$ is a pointed morphism in $\CC$.
\par 
Indeed, observe first that $A^E$  is canonically a pointed object with the pointed morphism $a^E \colon \varepsilon \to A^E$ given by the universal property of fibered products: the component morphisms are $a \colon \varepsilon \to A=A^{\{g\}}$ for $g \in E$. Similarly, we have a pointed object $(A^F, a^F)$ and the canonical projection $\pi\colon A^F \to A^E$ which verifies tautologically $\pi \circ a^F = a^E$. 
\par 
We introduce the following class of cellular automata over a pointed object in a concrete category with fibered products and a terminal object. 
\begin{definition}
\label{d:ca-concrete} 
Let $G$ be a group and let $\CC$ be a concrete category with fibered products and a distinguished terminal object.
For every pointed object $(A,a) \in \CC$, the class $CA_{\CC}(G,A)$ of \emph{$\CC$-cellular automata} consists of cellular automata $\tau \colon A^G \to A^G$ admitting for some finite memory set $M \subset G$ a local defining set map $\mu \colon A^M \to A$  which is the underlying map of a pointed morphism $(A^M,a^M) \to (A,a)$. 
\end{definition} 
\par 
A related notion of cellular automata over concrete categories were investigated in \cite{csc-concrete} where the authors obtain as the main result a reversibility and invertibility theorem  for cellular automata over concrete categories over residually finite group universes. In Section~\ref{s:generalization}, we will formulate and prove a more general result over LEF-group universes (Theorem~\ref{t:genralization-Hopfian}).  
\par 
We have the following simple lemma which results from the universal property of fibered products. 
\begin{lemma}
\label{l:induced-morphism-tau}
Let $G$ be a group and let $\CC$ be a concrete category  with a terminal object. Let $(A,a)\in \CC$ be a pointed object and let $\tau \in CA_{\CC}(G,(A,a))$ with a given memory set $M \subset G$. Then for every finite subset $E \subset G$, the induced set map $\tau_E^+ \colon A^{EM} \to A^E$ (Section~\ref{s:induced-map}) is a pointed morphism in $\CC$.  
\end{lemma}

\begin{proof}
Let $\varepsilon \in \CC$ be the terminal object. Observe first that $\tau_E^+ \colon A^{EM} \to A^E$ is a morphism  defined by the universal property of fibered products by the collection of morphisms $(T_{g})_{g \in E}$ where the component morphism $T_{g} \colon A^{E M} \to A^{\{g\}}$, $g \in E$, is the composition of the canonical projection $A^{E M} \to A^{gM}$ followed by the  morphism $A^{gM} \to A^{\{g\}}$ induced by $\mu$ via the the reindexing bijection $M \to gM$, $h \mapsto gh$. 
\par 
To conclude, it suffices to note that $T_g$ is a pointed morphism as the composition of the pointed morphisms $A^{E M} \to A^{gM}$, $x \mapsto x\vert_{gM}$ and $\mu$. It follows that $\tau_E^+$ is indeed a pointed morphism and the proof is complete. 
\end{proof}

\section{Generalizations} 
\label{s:generalization}

We establish the following general direct finiteness property and invertibility  property for cellular automata with finite product Hopfian or finite product co-Hopfian alphabets in a concrete category. 

\begin{theorem}
\label{t:genralization-Hopfian}
Let $G$ be a LEF-group and let $\CC$ be a concrete category with a terminal object and with fibered products. Let $(A,a) \in \CC$ be a pointed object. 
Let $\tau, \sigma  \in~CA_{\CC}(G,(A,a))$ then the following hold: 
\begin{enumerate}[\rm (i)] 
\item if $\tau$ 
is reversible and if the alphabet $A$ is finite product Hopfian then one has  $\tau^{-1} \in CA_{\CC}(G,(A,a))$; 
\item 
if $\sigma \circ \tau = \Id$ and if $A$ is finite product Hopfian or finite product co-Hopfian then one also has $\tau \circ \sigma = \Id$. 
\end{enumerate} 
\end{theorem} 

\begin{proof}[Proof of Theorem~\ref{t:genralization-Hopfian}.(i)] 
Suppose first that $A$ is a finite product Hopfian pointed object and that  $\tau$ is a reversible cellular automaton. Then the inverse $\tau^{-1}\colon A^G \to A^G$ is also a cellular automaton. 
\par 
It follows that we can choose a finite symmetric set $M \subset G$ containing $1_G$ such that $M$ is a memory set of both $\tau$ and $\tau^{-1}$. 
Let $\mu \colon A^M \to A$ be the local defining pointed morphism of $\tau$. Similarly, we have a local defining set map $\eta \colon A^M \to A$ of $\tau^{-1}$. Our goal is to prove that $\eta$ is a pointed morphism in the category $\CC$.   
\par 
By Lemma~\ref{l:induced-morphism-tau}, we can define a pointed  morphism $\tau_M^+ \colon A^{M^2} \to A^M$ in $\CC$  given by $\tau_M^+(z)(g)= \mu((g^{-1}z)\vert_M)$ for all $z \in A^M$ and $g \in M$. Let $\pi\colon  A^{M^2} \to A^{\{1_G\}}$ be the canonical  projection $v \mapsto v(1_G)$.  Since $\tau^{-1} \circ \tau=\Id$, we infer from Lemma~\ref{l:direct-local} that $\eta\circ \tau_M^+=\pi$ as set maps. 
\par 
Now, since $G$ is a LEF-group, we can find a finite group $F$ and an embedding  $\varphi \colon M^2 \to F$ such that $\varphi(ab)= \varphi(a) \varphi(b) $ for all $a,b \in M$. 
\par 
Moreover, up to  restricting to the subgroup of $F$ generated by the images $\varphi(a)$ for $a \in M$, we can clearly suppose that $F$ is generated by the finite set $E=\{\varphi(a) \colon a \in M\}=\varphi(M)$. 
\par 
Observe that since $M$ is symmetric and contains $1_G$ and since $\varphi$ is an embedding, we find that $E$ is also a finite symmetric subset of $F$ and that $1_F=\varphi(1_G) \in E$. 
\par 
Consider the pointed morphism $f \colon A^E \to A$ defined by $f(x)= \mu(y)$ for every $x \in A^E$ and $y\in A^M$ where $y(g)= x(\varphi(g))$ for $g \in M$. 
The morphism $f$ in turn induces by the universal property of fibered products  an $F$-equivariant pointed morphism
\[
\alpha \colon A^F \to A^F
\] 
defined by 
$\alpha(x)(h)= f((h^{-1}x)\vert_E)$ for all $x \in A^F$ and $h \in F$. 
\par 
Similarly, the map $\eta$ induces a set map $\psi \colon A^E \to A$ via the reindexing  bijection $M \to E$ given by $p \mapsto \varphi(p)$  and determines another $F$-equivariant set map 
\[
\beta \colon A^F \to A^F
\] 
defined by $\beta(x)(h)= \psi((h^{-1}x)\vert_E)$ for all $x \in A^F$ and $h \in F$. 
\par 
We are going to prove that $\beta \circ \alpha = \Id$. For this,  
consider the set map  $\alpha_E^+ \colon A^{E^2} \to A^E$ given by $\alpha_E^+(x)(h)= f((h^{-1}x)\vert_E)$ for all $x \in X^E$ and $h \in E$. Since $f$ is a pointed morphism in $\CC$, we deduce from the universal property of fibered products that $\alpha_E^+$ is also a pointed morphism in $\CC$. 
\par 
Let $\rho \colon X^{E^2} \to X^{\{1_F\}}$ be the canonical projection $u \mapsto u(1_G)$. 
Since $\varphi$ embeds $M^2$ into $F$ and since $\eta\circ \tau_M^+=\pi$, we have $\psi \circ \alpha_E^+=\rho$. 
\par 
Therefore, we find for every $x \in A^F$ and $h \in F$ that: 
\begin{align*}
    (\beta(\alpha(x)))(h) & = \psi((h^{-1}\alpha(x))\vert_E)\\
    & = \psi((\alpha(h^{-1}x))\vert_E)\\
    &=\psi(\alpha_E^+((h^{-1}x)\vert_{E^2}))\\ 
    & = (h^{-1}x)(1_F)\\
    &= x(h).
\end{align*}
\par 
Thus, we can conclude that $\beta \circ \alpha= \Id$ as a set map. It follows in particular that $\alpha\colon A^F \to A^F$ is a monomorphism in the category $\CC$. Indeed, suppose that $\delta_1, \delta_2 \colon B \to A^F$ are morphisms in $\CC$ such that $\alpha \circ \delta_1= \alpha \circ \delta_2$. Then since $\beta \circ \alpha= \Id$, we deduce the following equalities of set maps: 
\begin{align*}
    \delta_1 = \Id \circ \delta_1 & = (\beta \circ \alpha) \circ \delta_1  
     = \beta \circ (\alpha \circ \delta_1) \\ 
     & = \beta \circ (\alpha \circ \delta_2) = (\beta \circ \alpha) \circ \delta_2= \Id \circ \delta_2 = \delta_2. 
\end{align*}
\par 
Consequently, we deduce that  $\delta_1=\delta_2$ as morphisms in $\CC$. Since $A$ is a finite product Hopfian object, the monomorphism $\alpha$ must be an automorphism. Therefore, we find that $\beta\colon A^F \to A^F$ is in fact a morphism in $\CC$. 
\par 
Let $\varepsilon \in \CC$ be the terminal object then $a \colon \varepsilon \to A$ is a pointed  morphism. Since $\varepsilon$ is the terminal object in $\CC$, we have a pointed morphism $\omega \colon A^M \to \varepsilon$. We can thus define by the universal property of fibered products a pointed morphism $\iota \colon A^M \to A^F=A^M \times A^{F \setminus M}$ given by the product of the identity morphism $\Id \colon A^M \to A^M$ and the composition $A^M \xrightarrow{\omega}  \varepsilon \xrightarrow{a^{F \setminus M}} A^{F \setminus M}$. 
\par 
We now consider the pointed  morphism  $\nu \colon A^M \to A$ given by the universal property of fibered products by the formula 
$\nu (x)=(\beta \circ \iota (y))(1_F)$ for all $x \in A^M$ and  $y\in A^E$. 
\par 
By construction, it is clear that $\eta= \nu$ as set maps and thus $\eta$ is a pointed morphism in $\CC$. 
We can finally conclude that $\tau^{-1} \in CA_{\CC}(G,(A,a))$. 
Therefore, the proof of   Theorem~\ref{t:genralization-Hopfian}.(i) is complete. 
\end{proof}

With a similar proof, we can complete the proof of Theorem~\ref{t:genralization-Hopfian} as follows. 

\par 
\begin{proof}[Proof of Theorem~\ref{t:genralization-Hopfian}.(ii)] 
Suppose now that $\sigma \circ \tau = \Id$ and that $A$ is finite product Hopfian or finite product co-Hopfian. \par 
We choose a finite symmetric set $M \subset G$ containing $1_G$ such that $M$ is a memory set of both $\tau$ and $\sigma$. Let $\mu, \eta \colon A^M \to A$ be respectively the pointed morphisms which induce the local defining maps of $\tau$ and $\sigma$. 
\par 
We consider also the pointed morphism $\tau_M^+, \sigma_M^+ \colon A^{M^2} \to A^M$ in $\CC$ by Lemma~\ref{l:induced-morphism-tau}. Denote by $\pi\colon  A^{M^2} \to A^{\{1_G\}}$ the canonical projection. Then since $\sigma \circ \tau=\Id$, we deduce from Lemma~\ref{l:direct-local} that $\eta\circ \tau_M^+=\pi$ as set maps. 
\par 
From this point, we can follow the exact same proof of Theorem~\ref{t:genralization-Hopfian}.(ii) to obtain pointed morphisms $\alpha, \beta \colon A^F \to A^F$ such that $\beta \circ \alpha=\Id$ where we keep the same notations and constructions. 
\par 
In particular, we find that $\alpha$ is a monomorphism as in the above proof. Observe also that $\beta$ is an epimorphism. Indeed, let  
$\gamma_1, \gamma_2 \colon A^F \to B$ be  morphisms in $\CC$ such that $\gamma_1 \circ \beta  = \gamma_2 \circ  \beta$. As $\beta \circ \alpha= \Id$, we find that: 
\begin{align*}
    \gamma_1 = \gamma_1 \circ \Id & =
    \gamma_1 \circ (\beta \circ \alpha)  
     =  (\gamma_1 \circ \beta) \circ \alpha  \\ 
     & =  (\gamma_2 \circ \beta) \circ \alpha  = 
      \gamma_2 \circ (\beta \circ \alpha) = \gamma_2 \circ \Id = \gamma_2. 
\end{align*}
\par 
Hence, $\beta$ is indeed an epimorphism. We claim that $\alpha$ and $\beta$ are automorphisms. Indeed, if $A$ is a finite product Hopfian object, then the monomorphism $\alpha$ must be an automorphism. Otherwise, $A$ is a finite product Hopfian object so  the epimorphism $\beta$ must be an automorphism. Therefore, we have proven in all cases that $\alpha$ and $\beta$ are automorphisms. 
\par 
Consequently, we have $\alpha \circ \beta = \Id$. We claim that $\mu \circ \sigma_{M}^+= \pi$. Since $\varphi$ embeds $M^2$ into $F$, it suffices to prove that $f \circ \beta_E^+= \rho$ where we recall that $\rho \colon X^{E^2} \to X^{\{1_F\}}$ is the canonical projection $u \mapsto u(1_G)$. 
\par 
Let $y \in A^{F}$ and let $x= y \vert_{E^2}$. Then we deduce from the equality $\alpha \circ \beta= \Id$ and from the constructions that: 
\begin{align*} 
f ( \beta_E^+(x)) = f(\beta(y)\vert_E) = \alpha(\beta(y))(1_G)=y(1_G)= \rho(x). 
\end{align*}
\par 
This proves the claim that $\mu \circ \sigma_{M}^+= \pi$ as set maps. Hence, Lemma~\ref{l:direct-local} implies that 
$\tau \circ \sigma=\Id$ and the proof is complete. 
\end{proof}

\section{Applications and examples} 
\label{s:applications}

In this section, we will give several explicit examples to illustrate our main results obtained in this paper. 

\subsection{$R$-module cellular automata}
Let $G$ be a group and let $A$ be a module over a ring $R$. Then the class $CA_{R-mod}(G,A)$ of \emph{$R$-module  cellular automata} consists of cellular automata $\tau \colon A^G \to A^G$ admitting 
a local defining map $\mu \colon A^M \to A$ which is a homomorphism of $R$-modules for some finite set $M \subset G$. 
\par 
Note that the category of $R$-modules is a concrete category with the trivial $R$-module $\{0_R\}$ as the terminal object. The classes of linear cellular automata and group cellular automata (cf.~Section~\ref{s:group-ca}) are important examples of $R$-module cellular automata (see, e.g.  \cite{ca-and-groups-springer},  \cite{phung-dcds},  \cite{phung-israel}, \cite{phung-post-surjective}). 
\par 
As a direct consequence of Theorem~\ref{t:genralization-Hopfian}, we obtain the following invertibility and direct finiteness result for the class $CA_{{R-mod}}$:
\begin{corollary}
\label{c:app-invertible-module} 
Let $G$ be a LEF-group and let $R$ be a  ring. Let $A$ be an $R$-module and let $\sigma, \tau \in CA_{R-mod}(G,A)$. If $A$ is a Noetherian $R$-module and if 
$\tau$ is reversible then one has $\tau^{-1} \in CA_{R-mod}(G,A)$. Moreover, if  $\sigma \circ \tau= \Id$ then one has $\tau \circ \sigma= \Id$ in each of the following cases: 
\begin{enumerate}[\rm (a)] 
\item 
$A$ is a Noetherian $R$-module or an Artinian $R$-module; 
\item 
$R$ is a commutative ring and $A$ is a finitely generated $R$-module.   
\end{enumerate} 
\end{corollary}

\begin{proof}
For the first statement, suppose that $A$ is a Noetherian $R$-module. Then $A^E$ is also a Noetherian $R$-module for every finite set $E$. Hence, $A$ is a finite product Hopfian $R$-module. Therefore, if $\tau$ is reversible, we deduce from Theorem~\ref{t:genralization-Hopfian}.(i) that $\tau^{-1} \in CA_{R-mod}(G,A)$. 
\par 
For the second statement, suppose that $R$ is a commutative ring and $A$ is a finitely generated $R$-module. Then we infer from Proposition~\ref{p:finite-product-commutative-ring-hopfian} that $A$ is a finite product co-Hopfian $R$-module. If $A$ is a Noetherian $R$-module then we have seen as above that $A$ is a finite product Hopfian $R$-module. Similarly, note that every Artinian $R$-module is a co-Hopfian $R$-module and direct sums of Artinian $R$-modules are Artinian $R$-modules. Hence, if $A$ is an Artinian $R$-module then it is a finite product co-Hopfian $R$-module.
\par 
We conclude that $A$ is a finite product Hopfian or co-Hopfian $R$-module in all cases (a) and (b). 
Consequently, Theorem~\ref{t:genralization-Hopfian}.(ii) implies that $\tau \circ \sigma= \Id$. The proof is thus complete.  
\end{proof}

\subsection{Group cellular automata} 
\label{s:group-ca}
Let $G$ be a group and let $A$ be a group. Then the class $CA_{grp}(G,A)$ of \emph{group cellular automata} consists of cellular automata $\tau \colon A^G \to A^G$ admitting 
a local defining map $\mu \colon A^M \to A$ which is a group homomorphism for some finite memory set $M \subset G$. 
\par 
In the spirit of the geometric generalization of Kaplansky's direct finiteness conjecture \cite{phung-geometric}, we obtain as another immediate application of Theorem~\ref{t:genralization-Hopfian}  the following direct finiteness result for group cellular automata: 
\begin{corollary}
\label{c:app-invertible-module} 
Let $G$ be a LEF-group and let $A$ be a Hopfian or a co-Hopfian group (see Examples~\ref{ex:hopfian}). Suppose that $\sigma, \tau \in CA_{grp}(G,A)$ satisfy $\sigma \circ \tau= \Id$. Then one also has $\tau \circ \sigma= \Id$. 
\end{corollary}
\begin{proof}
It is a direct application of Theorem~\ref{t:genralization-Hopfian}. 
\end{proof}

\bibliographystyle{siam}

\end{document}